\documentclass{amsart}
\usepackage[english]{babel}

\newcommand{\Ps}{{\mathbf{P}}}
\newcommand{\Z}{{\mathbf{Z}}}
\newcommand{\C}{{\mathbf{C}}}

\newcommand{\va}{\mathbf{a}}
\newcommand{\vb}{\mathbf{b}}

\newcommand{\vc}{\mathbf{c}}
\newcommand{\vd}{\mathbf{d}}

\renewcommand{\phi}{\varphi}
    \newtheorem{lemma}{Lemma}[section]
    \newtheorem{proposition}[lemma]{Proposition}
    \newtheorem{theorem}[lemma]{Theorem}
    
    \newtheorem{corollary}[lemma]{Corollary}

   \theoremstyle{definition}
    \newtheorem{definition}[lemma]{Definition}

    \newtheorem{remark}[lemma]{Remark}
    
    \DeclareMathOperator{\rank}{rank}

\DeclareMathOperator{\sing}{{sing}}

\DeclareMathOperator{\prim}{{prim}}

\DeclareMathOperator{\MW}{{MW}}

\DeclareMathOperator{\coker}{{coker}}

\usepackage[all]{xy}
\begin{document}
\title[Mordell-Weil rank and  syzygies]{Cuspidal plane curves, syzygies and a bound on the MW-rank}
\author{Remke Kloosterman}
\address{Institut f\"ur Mathematik, Humboldt-Universit\"at zu Berlin, Unter den Linden 6, D-10099 Berlin, Germany}
\email{klooster@math.hu-berlin.de}
\thanks{The author thanks Alberto Calabri, Alexandru Dimca, David Eisenbud, Anatoly Libgober, Matthias Sch\"utt and Orsola Tommasi for several comments on a previous versions of this paper. The author thanks the referee for many useful comments.
The author acknowledges the partial support from the DFG under research grant KL 2244/2-1.}

\begin{abstract} Let $C=Z(f)$ be a reduced plane curve of degree $6k$, with only nodes and ordinary cusps as singularities. Let $I$ be the ideal of the points where $C$ has a cusp. Let $\oplus S(-b_i)\to \oplus S(-a_i) \to S\to S/I$ be a minimal resolution of $I$. We show that $b_i\leq 5k$. From this we obtain that the Mordell-Weil rank of the elliptic threefold $W:y^2=x^3+f$ equals $2\#\{i\mid b_i=5k\}$.
Using this we find an upper bound for  the Mordell-Weil rank of $W$, which  is $\frac{1}{18} (125+\sqrt{73}-\sqrt{2302-106\sqrt{73}})k+l.o.t.$  and we find an upper bound for  the exponent of $(t^2-t+1)$ in the Alexander polynomial of $C$, which is   $\frac{1}{36}(125+\sqrt{73}-\sqrt{2302-106\sqrt{73}})k+l.o.t.$. This improves a recent bound of Cogolludo and Libgober almost by a factor 2.
\end{abstract}
\date{\today}

\maketitle

\section{Introduction}\label{secInt}
In this paper we study reduced plane curves $C$ of degree $d=6k$  having only nodes and ordinary cusps as singularities. We allow $C$ to be reducible.
Let $z_0,z_1,z_2$ be coordinates on $\Ps^2$, and let $S=\C[z_0,z_1,z_2]$. Let $f\in S_{6k}$ be an equation for $C$. Let $\Sigma$ be the set of cusps of $C$ (we will ignore the nodes).

 Consider now the elliptic threefold defined by
\[ Z(-y^2+x^3+f)\subset \Ps(2k,3k,1,1,1). \]
Let $\MW(\pi)$ be the Mordell-Weil group, i.e., the group of rational sections of the elliptic fibration. It is known that the rank of $\MW(\pi)$ can be expressed in terms of the geometry of $C$ (see Proposition~\ref{prpLinSys}), namely
\[ \rank \MW(\pi)= 2 \dim \coker \left( S_{5k-3} \stackrel{\oplus ev_p}{\longrightarrow} \oplus_{p\in \Sigma} \C\right).\]

We can also consider the fundamental group $\pi_1(\Ps^2\setminus C)$. With this group we can associate the so-called Alexander polynomial of $C$. It turns out that the exponent of $(t^2-t+1)$ in the Alexander polynomial equals
\[  \dim \left(\coker S_{5k-3} \stackrel{\oplus ev_p}{\longrightarrow} \oplus_{p\in \Sigma} \C\right).\]
Hence both invariants coincide. Cogolludo and Libgober \cite{CogLib} noticed this and proved for a much larger class of singular plane curves that the degree of the Alexander polynomial is related with the Mordell-Weil group of an associated elliptic fibration.

In this paper we give a non-trivial upper bound $g(k)$ for the Mordell-Weil rank, which also yields an upper bound  for the exponent of $t^2-t+1$ in the  Alexander polynomial. Asymptotically we have that
\begin{equation}\label{rnkbnd} \lim_{k\to \infty} \frac{g(k)}{k}=\frac{1}{18} (125+\sqrt{73}-\sqrt{2302-106\sqrt{73}}) \approx 5.34\end{equation}
The best known previous upper bound for the exponent of $(t^2-t+1)$ in the Alexander polynomial of a cuspidal curve seems to be due to Cogolludo and Libgober \cite{CogLib}, and equals $5k-1$ (this implies that the MW-rank is at most $10k-2$). This bound is an immediate consequence of the Shioda--Tate formula. The divisibility theorem for the Alexander polynomial of Libgober yields an upper bound of $6k-2$ for the exponent in the Alexander polynomial.

Our bound is deduced from two other bounds.
Suppose we fix $r,k$ and look for $C_{2r,k}$ the minimal number of cusps on a degree $6k$ curve such that the corresponding elliptic fibration has Mordell-Weil rank at least $2r$.
We show that
\begin{equation}\label{mainbnd} C_{2,k}=6k^2 \mbox{ and } C_{2r,k}\geq 6k^2+3(r-1)k-\frac{3}{4}r(r-1)+O\left(\frac{1}{k}\right) \mbox{for } k \to \infty.\end{equation}
The number of cusps on a degree $d$ curve can be bounded by $\frac{125+\sqrt{73}}{432} d^2- \frac{511+ 11\sqrt{73}}{1752} d$ (see \cite{Langer}). Combining both bounds yields the upper bound (\ref{rnkbnd}).
This bound is very unlikely to be sharp. We expect that $C_{2r,k}$ can be bounded from below by a function of the form $h(r)k^2$, where $h$ is increasing in $r$, rather than constant. However, the bound for $C_{2,k}$ is sharp. If we take general polynomials $f_1\in S_{2k}$ and $f_2\in S_{3k}$, and set $f=f_1^3+f_2^2$, then $f$ has $6k^2$ cusps and the Mordell-Weil rank is at least 2. The fact that $C_{2,k}\geq 6k^2$  holds, can also be obtained by different methods, namely if $C$ has less then $6k^2$ cusps then  $\pi_1(\Ps^2\setminus C)$ is abelian and therefore $C$ has constant Alexander polynomial. In particular, the Mordell-Weil rank is zero in this case, see \cite{ZarSyz}.

The main idea of the proofs is to consider the resolution of the ideal $I$ of $\Sigma$:
\[ 0 \to \oplus_{i=1}^t S(-b_i) \to \oplus_{i=1}^{t+1} S(-a_i) \to S\to S/I\to 0.\]
There are several restrictions on $a_i,b_i$ coming from the fact that $I$ is the ideal of a finite set of points in $\Ps^2$. These restrictions are classically known, see Proposition~\ref{prpFiniteSet}. 
We find further restrictions on the $a_i$ and $b_i$ by a combination of Bezout's theorem and an upper bound for the number of cusps on a degree $d$ plane curve. 
See Proposition~\ref{prpSingLocu}.

Using specializations to elliptic surfaces we show that $b_i\leq 5k$ for all $i$. Using an expression for the difference between the Hilbert polynomial of $I$ and the Hilbert function of $I$ we obtain that $\rank \MW(\pi) =2\#\{i \mid b_i=5k\}$. This fact is proved in Proposition~\ref{prpSyzForm}.


 After distributing a preliminary version of this paper we learned the following:
Zariski \cite{ZarSyz} proved that the Castelnuovo-Mumford regularity of the cuspidal locus of an \emph{irreducible} plane curve is at most $5k-1$. This is done by studying the cyclic degree $6k$ cover of $\Ps^2$ ramified along the curve $C$. The statement on the regularity  implies that  $b_i\leq 5k$ holds in the case of irreducible curves. In the case of reducible curves Zariski proved that the regularity of the cuspidal locus is at most $6k-2$. 

The statement $b_i\leq 5k$ in the case of reducible curves seems  to be known to the experts, although we could not identify a proof for this statement in the literature. The techniques to extend Zariski's proof to the reducible case have been around since the beginning of the 1980s (\cite{EsnMil},\cite{LibAlexArc}). However, our proof is different from the existing proofs in the literature.

There are other examples where the highest degree syzygies of the ideal of the singular locus has a geometric interpretation. E.g., if we consider a minimal resolution of the ideals of the nodes, then a syzygy has degree at most the degree of the curve, and the number of highest degree syzygies is one less than the number of irreducible components of the curve. (Proposition~\ref{prpSyzNode})

In Section~\ref{secBnd} we prove the bound (\ref{mainbnd}) under an extra technical assumption on the  $a_i$ and $b_i$: After permuting the $a_i$ and $b_i$ we may assume that  the $a_i$ and $b_i$ both form a descending sequence. A priory, we know that $b_i>a_i$. In Section~\ref{secBnd} we assume that  $a_i\leq b_{i+1}$ for all $i$.

In Section~\ref{secAmi} we study the case when there is some $i$ with $a_i>b_{i+1}$.
We consider the ideal generated by all generators of $I$ of degree less than $a_i$. This ideal defines  a subscheme of $\Ps^2$ which is the union of a (possibly non-reduced) curve and a zero-dimensional scheme. We can analyze this situation in similar way as above and obtain further restrictions on the $a_i$ and the $b_i$. These further obstructions allow us to construct a new sequence $a'_i$, $b'_i$ that corresponds with an ideal $I'$ such that $\#Z(I')\leq \#Z(I)$ and the $a'_i,b'_i$ satisfy the extra condition used in Section~\ref{secBnd}. Hence also in this case the lower bound (\ref{mainbnd}) holds.

\section{Resolution of the locus of cusps}\label{secRes}

We cite first a result on the resolution of the ideal of finitely many points in $\Ps^2$.
\begin{proposition}\label{prpFiniteSet} Let $I$ be the ideal of finitely many distinct points in $\Ps^2$. Then $I$ has a free resolution
\begin{equation}\label{eqnstdres} 0\to \oplus_{i=1}^t S(-b_i) \to \oplus_{i=1}^{t+1} S(-a_i) \to S\to S/I \to 0,\end{equation}
such that 
\begin{enumerate}
\item for all $i$ we have that  $a_i,b_i\in \Z$ and $a_i>0, b_i>0$;
\item\label{numb} $\sum_{i=1}^{t+1} a_i=\sum_{i=1}^t b_i$;
\item for $i=1,\dots t$ we have $b_i>a_i\geq a_{i+1}$ and for $j=1,\dots,t-1$ we have $b_j\geq b_{j+1}$.
\item  
$ \# Z(I) = \frac{1}{2}\left(\sum b_i^2-\sum a_i^2\right) =B(s)-\sum_{i=1}^{t+1} B(s-a_i)+\sum_{i=1}^t B(s-b_i)$ for every $s$, where  $B(s)=\frac{1}{2}(s+1)(s+2)$.
\end{enumerate}
\end{proposition}
\begin{proof} This follows almost immediately from the fact that $I$ has a free resolution of length $1$ and that the Hilbert polynomial of $I$ is constant.
See \cite[Section 3.1]{EisSyz}.
\end{proof}



\begin{definition}
For $d$ a positive integer, define $M(d)$ to be the maximal number of ordinary cusps on a (possibly reducible) degree $d$ curve. 
\end{definition}
\begin{remark}\label{rmkMiy}
The best known asymptotic upper bound for $M(d)$ we are aware of  is a bound obtained by Langer (see \cite[Section 11]{Langer}). Langer states that \[\limsup_{d\to \infty} M(d)/d^2\leq \frac{125+\sqrt{73}}{432}.\] However, a closer inspection of the proof reveals that
\[ M(d) \leq \frac{125+\sqrt{73}}{432} d^2- \frac{511+ 11\sqrt{73}}{1752} d.\]
On the other side, Hirano showed that \cite[Corollary 3]{Hirano}
\[ \limsup_{d\to \infty} \frac{ M(d)}{d^2}\geq \frac{9}{32}.\]
\end{remark}
\begin{proposition}\label{prpSingLocu} Suppose $I$ is the ideal of the locus of cusps of a plane curve $C$ of degree $d$. Then
\[ \# Z(I) \leq \min\left( \frac{a_{t+1} d}{2},M(d) \right) \leq \min \left(a_{t+1},\frac{5}{8}d-\frac{3}{4}\right) \frac{d}{2} \]
\end{proposition}

\begin{proof}
The upper bound $\# Z(I)\leq M(d)$ is obvious. We prove now that $\# Z(I)\leq \frac{a_{t+1}d}{2}$.

Since the ideal $I$ contains an element of degree $a_{t+1}$, there exists a curve $C'$ of degree $a_{t+1}$ such that  all cusps of $C$ are points of $C'$. By Bezout's theorem we have that if $2\# Z(I)> a_{t+1}d$ then $C$ and $C'$ have a common component $C''$ of degree $d''\leq a_{t+1}$. Since cusps are irreducible singularities it follows that all the cusps of $C$ that are also points of $C'$ are actually cusps of $C''$. Hence 
\[ \# Z(I) \leq M(d'')+\frac{1}{2} (d-d'')(a_{t+1}-d'')\leq \frac{5}{16}(d'')^2 +\frac{1}{2}a_{t+1}d+\frac{1}{2}(d'')^2-\frac{1}{2}d''(a_{t+1}+d)
\]
Combining this with $\# Z(I)\geq \frac{1}{2}da_{t+1}$ yields
\[ 0 \leq \frac{13}{16} (d'')^2-\frac{1}{2}d''(a_{t+1}+d).\]
Dividing this inequality by $d''/2$ and using  $2d''\leq a_{t+1} + d$ yields
\[ 0 \leq \frac{13}{8} d''-(a_{t+1}+d)\leq \frac{-3}{8}d''. \]
Hence the degree of $C''$ is 0. Equivalently, the component $C''$ does not exist and $\# Z(I)\leq \frac{1}{2}a_{t+1}d$.
\end{proof}

\section{Syzygies and MW-rank}\label{secMW}
Let $S:=\C[z_0,z_1,z_2]$ be the polynomial ring in three variables. Let $S_d$ be the subspace of homogeneous polynomials of degree $d$. 

Fix an integer $k$ and a square-free polynomial $f\in S_{6k}$, such that the plane curve $C=Z(f)$ has only nodes and ordinary cusps as singularities. Let $\Sigma$ denote the set of cusps of $C$. Let $I\subset S$ be the ideal of $\Sigma$.

Let $W_f\subset \Ps(2k,3k,1,1,1)$ be the hypersurface given by the vanishing of
\[ -y^2+x^3+f\]
The threefold $W_f$ is birational to an elliptic threefold $\pi: X\to R$, where $R$ is a rational surface and the elliptic fibration $\pi$ is birational to the  projection $\psi:W_f\setminus \{(1:1:0:0:0)\} \to \Ps^2$ from $(1:1:0:0:0)$ onto the plane $\{x=y=0\}$. The explicit construction of $\pi$ is slightly complicated, see \cite{MirEllThree}.
For $p\in \Ps^2$ the Zariski closure of $\psi^{-1}(p)$ is either an elliptic curve with $j$-invariant 0 or a cuspidal cubic, depending on whether $p\in C$ or not.

The Mordell-Weil group $\MW(\pi)$ of $\pi$ is the group of rational sections of $\pi$. This is a finitely generated group, and if the singularities of $C$ are ``mild'' then one has an algorithm to compute the rank of $\MW(\pi)$, see \cite{ell3HK}: 
\begin{proposition}\label{prpLinSys} We have the following equality
\begin{equation}\rank MW(\pi) = 2\dim \left(\coker S_{5k-3} \stackrel{ev_P}{\longrightarrow} \oplus_{p\in \Sigma} \C\right).\end{equation}                                                                                                                 
 \end{proposition}

\begin{proof}
For the case $k=1$ see \cite[Section 9]{elljcst}. The general case follows along the same lines: 

An $A_1$ singularity of $C$ yields an $A_2$ singularity of $W_f$, whereas an $A_2$ singularity of $C$ yields a $D_4$ singularity on $W_f$.

Let $\Sigma$ be the singular locus. We can now compute $H^4(W_f)_{\prim}$ as the cokernel of
\[ H^4(\Ps(2k,3k,1,1,1)\setminus W_f) \cong H^3(W_f\setminus \Sigma)\to H^4_\Sigma(W_f).\]
An $A_2$ singularity of $W_f$ does not contribute to $H^4_{\Sigma}$, whereas a $D_4$ singularity does \cite[Example 1.9]{DimBet}. Actually, using the ideas from \cite[Section 1]{DimBet} it follows that $H^4_p(W_f)=\C(-2)^2$ if $p$ is a $D_4$ singularity of $W_f$, hence $H^4_\Sigma(W_f)$ is of pure Hodge type $(2,2)$. Then by the main results of \cite{ell3HK} we have $\rank \MW(\pi)=h^4(W_f)-1$. 

Let $\omega$ be a third root of unity. The map $\omega:[x:y:z_0:z_1:z_1]\mapsto[\omega  x: y:z_0:z_1:z_2]$ is an automorphism of $W_f$ and fixes every point of $\Sigma$. The map $H^4(\Ps(2k,3k,1,1,1)\setminus W_f)\to H^4_{\Sigma}(W_f)$ is $\omega^*$-equivariant, so we may decompose it in a $\omega$ and $\omega^2$ eigenspace, which have both the same dimension, and a $1$-eigenspace, which is trivial.

The argument used in \cite[Section 9]{elljcst} shows in this case that the $\omega$-eigenspace of the co-kernel of  $H^4(\Ps(2k,3k,1,1,1)\setminus W_f)\to H^4_{\Sigma}(W_f)$ has dimension
\[\dim \left(\coker S_{5k-3} \stackrel{ev_P}{\longrightarrow} \oplus_{p\in \Sigma} \C\right).\]
\end{proof}

\begin{remark} With the fundamental group of $\Ps^2\setminus C$
one can associate the so-called Alexander polynomial. Cogolludo and Libgober \cite{CogLib} showed that the exponent of $t^2-t+1$ in this polynomial equals half the rank of $\MW(\pi)$.

Hence each statement which we make on the rank of $\MW(\pi)$ is also a statement  on the exponent of $t^2-t+1$ in the Alexander polynomial of a cuspidal curve.
\end{remark}

We will calculate $\rank \MW(\pi)$ using a projective resolution of $I$ and use properties of the resolution to bound the Mordell-Weil rank in terms of $k$.
Two more or less obvious restrictions on the resolution come from Bezout's theorem (Proposition~\ref{prpSingLocu}) and the upper bound for the maximal number of cusps on a degree $d$ plane curve (Remark~\ref{rmkMiy}). The third restriction comes from considerations on the Mordell-Weil rank of elliptic surfaces.
 
\begin{lemma}\label{lemDefFor} Let $J\subset S$ be the ideal of a zero-dimensional scheme. Let
\[ 0\to \oplus_{i=1}^t S(-b_i) \to \oplus_{i=1}^{t+1} S(-a_i) \to S\to S/J \to 0\]
be a minimal resolution of $S/I$. Let $n$ be integer and set $b'_i:=b_i-n, a'_i:=a_i-n$. Then the defect of the linear system of degree $n-3$ polynomials vanishing along $Z(J)$  equals
\[\left( \sum_{i\mid b_i'\geq 0}  (b'_i+1)(b'_i+2)- \sum_{i\mid a_i'\geq 0} (a'_i+1)(a'_i+2)\right).\]
\end{lemma}

\begin{proof}
 The defect of the linear system is precisely the difference between the Hilbert function of $I$  in degree $n-3$ and the Hilbert polynomial of $I$.
Let $B(s)$ be as in the previous section. Then Hilbert polynomial $P_J(s)$ equals
\[ B(s)+\sum_{i=1}^t B(s-b_i)-\sum_{i=1}^{t+1} B(s-a_i).\]
The Hilbert function $h_J(s)$ evaluated at $s$ equals
\[B(s)+\sum_{i \mid b_i\leq s} B(s-b_i) -\sum_{i\mid a_i \leq s} B(s-a_i). \]
Since $B(-1)=B(-2)=0$ we may replace the condition ``$\leq s$" by ``$\leq s+2$" in the above formula. Hence 
\[ p_I(s)-h_I(s)=\sum_{i \mid b_i\geq s+3} B(s-b_i) -\sum_{i\mid a_i \geq s+3} B(s-a_i).\] 
Substituting $s=n-3$ in the above formula finishes the proof.
\end{proof}

\begin{lemma}\label{lemBnd} Let $J\subset S$ be the ideal of a zero-dimensional scheme. Let
\[ 0\to \oplus_{i=1}^t S(-b_i) \to \oplus_{i=1}^{t+1} S(-a_i) \to S\to S/J \to 0\]
be a minimal resolution of $S/J$. Suppose there is an integer $n$ and  a linear function $g(m)$ such that for all positive integers $m$ and for a general choice of three homogeneous polynomials $g_0,g_1,g_2\in S_m$  the defect $\delta_m$ of the linear system of degree $mn-3$ polynomials vanishing along $Z(\varphi^*(I))$ is at most $g(m)$, where $\varphi:\Ps^2\to  \Ps^2$ is given by $(z_0:z_1:z_2)\mapsto (g_0:g_1:g_2)$.

Then $b_i\leq n$, $a_i<n$ and $\delta_m=\#\{i\mid b_i=n\}$.
\end{lemma}

\begin{proof}
Define $a_i'$ and $b_i'$ as in Lemma~\ref{lemDefFor}.
Suppose that for some $j$ we have $b'_j>0$. Fix a positive integer $w$ such that 
\[\begin{array}{ll} (b'_j w+1)(b'_j w+2)>g(w) &\mbox{ if } a'_j<0 \mbox{ or } \\(b'_j w+1)(b'_j w+2)-(a'_j w+1)(a'_j w+2)>g(w) & \mbox{ if } a_j' \geq 0.\end{array}\]

Let $\varphi:\Ps^2\to \Ps^2$ be a general map of degree $w$. Then the resolution of $S/\varphi^*(I)$ 
\[ 0 \to \oplus_{i=1}^t S(-b_iw)\to \oplus_{i=1}^{t+1} S(-a_iw)\to S\to S/I\to0.\]
In particular, by the previous lemma we obtain
\[ \delta_{w}=\left( \sum_{i\mid b_i' \geq 0}  (b'_iw+1)(b'_iw+2)- \sum_{i\mid a_i'\geq 0} (a'_iw+1)(a'_iw+2) \right)>g(w).\]
This contradicts $\delta_{w}\leq g(w)$. Hence $b'_i\leq 0$, for all $i$ and $a'_i<0$ for all $i$.

Applying  Lemma~\ref{lemDefFor} again yields that for all $m$ we have $\delta_m=\#\{i\mid b_i=n\}$.
\end{proof}

We return now to the case where $C$ is a cuspidal curve and $I$ is the ideal of the points where $C$ has a cusp.
\begin{proposition}\label{prpSyzForm} Let
\[ 0\to \oplus_{i=1}^t S(-b_i) \to \oplus_{i=1}^{t+1} S(-a_i) \to S\to S/I \to 0\]
be a resolution of $I$. Then for all $i$ we have that $b_i\leq 5k$, $a_i<5k$ and
\[ \rank \MW(\pi)= 2 \#\{i\mid b_i=5k\}.\]
\end{proposition}

\begin{proof}
Take a general line $\ell\subset \Ps^2$. The (projective) surface $\overline{\pi^{-1}(\ell)}\subset \Ps(2k,3k,1,1,1)$ might be singular. Denote with $\widetilde{\pi^{-1}(\ell)}$ a resolution of singularities of this surface. This surface admits a natural elliptic fibration $\pi_\ell:\widetilde{\pi^{-1}(\ell)}\to \ell$.  From the theory of elliptic surface  we obtain the following well-known inequalities:
\begin{equation} \label{rnkineq} \rank \MW(\pi)\leq  \rank \MW(\pi_\ell) \leq h^{1,1}(\widetilde{\pi^{-1}(\ell)})-2 = 10k-2.\end{equation}
The first inequality is a standard result on specializations. The second inequality follows from the Shioda-Tate formula. The final equality is a well-known fact for elliptic surfaces, see e.g., \cite{MiES}.

Now take three general polynomials $g_0,g_1,g_2$ of degree $w$. 
Let $\varphi: \Ps^2\to \Ps^2$ be the map defined by $\varphi(z_0:z_1:z_2)= (g_0:g_1:g_2)$.
Let $\tilde{f}=\varphi^*(f)\in S_{6kw}$ and let $\pi_w:X_w\to R'$ be the pull-back of the elliptic fibration $\pi:X\to R$.
 For general $g_i$ the curve defined by $\tilde{f}$ has only nodes and ordinary cusps as singularities and the locus  $\tilde{\Sigma}$ consisting of the cusps of $Z(\tilde{f})$ equals $\varphi^{-1}(\Sigma)$. In particular, the corresponding ideal $\tilde{I}$ has the following minimal free resolution
\[ 0 \to \oplus_{i=1}^t S(-b_iw)\to \oplus_{i=1}^{t+1} S(-a_iw)\to S\to S/\tilde{I}\to0.\]

From Proposition~\ref{prpLinSys} it follows that the rank of $\MW(\pi)$ is twice the dimension of the cokernel of the evaluation map
$S_{5k-3} \stackrel{ev_P}{\to} \oplus_{p\in \Sigma} \C$, which equals the defect of the linear system of degree $5k-3$ polynomials through $Z(I)$. Similarly the Mordell-Weil rank of $\pi_w$ equals the defect of the linear system of degree $5kw-3$ polynomials through $Z(\varphi^*I)$. This defect is bounded by $10kw-2$, which is linear in $w$. Hence we can apply Lemma~\ref{lemBnd} and we obtain that $b_i\leq 5k, a_i<5k$ and $\rank \MW(\pi)=2\#\{i\mid b_i=5k\}$.
\end{proof}

If $C'$ is an \emph{irreducible} curve of degree $d$ and $I'$ the ideal of the points of $C'$ where  $C'$ has a node then each syzygy of $I'$ has degree at most $d-1$ (this is implied by the exercises 24 and 31 of  \cite[Appendix A]{ACGH}). Actually, a statement analogous to Proposition~\ref{prpSyzForm} holds for the locus of nodes of a plane curve, but we  could not find this particular result in the literature.

\begin{proposition}\label{prpSyzNode} Let $C'\subset \Ps^2$ be a reduced plane curve of degree $d$ with only nodes and ordinary cusps as singularities. Let $c$ be the number of irreducible components of $C'$. Define $\mathcal{N}$ to be the locus of nodes of $C'$. Let $I'$ be the ideal of $\mathcal{N}$ and
\[ 0\to \oplus S(-b_i) \to \oplus S(-a_i) \to S \to S/I' \to 0\]
be a minimal resolution of $I'$. Then $b_i\leq d$ and
\[ \#\{i\mid b_i=d\}=c-1.\]
\end{proposition}

\begin{proof}
From the Mayer--Vietoris sequence it follows easily that  $h^2(C')=c$.  
We would like to use Dimca's method \cite{DimBet} to calculate $h^2(C')$ in terms of the defect of a linear system. However, the previous discussion was based on the results from \cite{DimBet}, and this text considers only hypersurfaces in weighted projective spaces of dimension at least 3. A variant that works for plane curves is presented in \cite[Page 201]{Dim}.

Let $\Sigma'=C'_{\sing}$. If $\Sigma'$ is empty then there is nothing to prove, so assume that $\Sigma'$ is non-empty. Let $C^*=C'\setminus \Sigma'$, let $\Ps^*=\Ps^2 
\setminus \Sigma'$ and $U=\Ps^* \setminus C^*=\Ps^2 \setminus C'$.

Let $p\in \Sigma'$, $V_p\subset \Ps^2$ be a small neighborhood of $p$ and $D_p=C'\cap V_p$. Then in \cite[Page 201]{Dim} it is shown that $H^2(C')_{\prim}$ equals the cokernel of
\[ H^2(U)(1)\to \oplus_{p\in \Sigma'} H^2(V_p\setminus D_p)\]
Following the discussion after \cite[Formula (1.7)]{DimBet} we obtain that $H^2(V_p\setminus D_p)$ is zero if $p$ is a cusp and is one-dimensional if $p$ is a node. In the nodal case we have that $H^2(V_p\setminus D_p)$ is spanned by 
$ \frac{1}{xy} dx\wedge dy.$

The map $H^2(U)(1)\to \oplus_{p\in \mathcal{N}} \C \frac{1}{xy} (dx\wedge dy)$ can be given explicitly as follows. Let $f\in S_d$ be a polynomial defining $C'$. Set
\[\Omega:=z_0z_1z_2 \left(\frac{dz_1}{z_1}\wedge\frac{dz_2}{z_2}-\frac{dz_0}{z_0}\wedge\frac{dz_2}{z_2}+\frac{dz_0}{z_0}\wedge\frac{dz_1}{z_1}\right) \]
Following \cite[Section 1]{DimBet} we have $F^3H^2(U)=0$,
\[  F^2H^2(U)\subset \left\{\frac{g}{f} \Omega \mid g\in S_{d-3}\right\}\]
and  $H^2(U)=F^0H^2(U)=F^1H^2(U)$. Moreover, the latter space is spanned by
\[\left\{\frac{g}{f} \Omega \mid g\in S_{d-3}\right\} \cup \left\{\frac{h}{f^2} \Omega \mid h\in S_{2d-3}\right\}.\]
Since $H^2(U)(1)\to H^2_{\Sigma'}(C)$ is a morphism of Hodge structures and the image has only classes of type $(1,1)$, it follows that $Gr^F_1 H^2(U)$ is mapped to zero in $H^2_{\Sigma'}(C')$. Hence to determine the co-kernel of $H^2(U)(1)\to \oplus \C \frac{1}{xy} (dx\wedge dy)$ we can restrict this map to $F^2H^2(U)$. From the local construction of this map it follows directly that $\frac{g}{f}\Omega$ is mapped to $g(p)\frac{1}{xy}(dx\wedge dy)$. Combining everything we obtain that
\begin{equation}\label{eqnDimNode}
c-1=h^2(C')_{\prim}=\dim \coker \left(S_{d-3}\to \oplus_{p\in \mathcal{N}} \C\right).
 \end{equation}

Hence $c-1$ equals the defect of the linear system of degree $d-3$ polynomial through the set of nodes of $C$. For a general degree $m$ base change, the pullback $C'_m$ of $C'$ is a nodal curve, and the ideal of the nodes of $C'$ is the pullback of $I$. The number of irreducible components of $C_m$, and hence the defect of the linear system of degree $md-3$ polynomials through the nodes, can be bounded by $md$. Hence we may apply Lemma~\ref{lemBnd} to  the minimal resolution of $I'$, and obtain that $b_i\leq d$, $a_i<d$ and $c-1=\#\{i\mid b_i=d\}$.
\end{proof}

\section{Upper bound for $\rank \MW(\pi)$ (strongly admissible case)}\label{secBnd}
As discussed in the proof of Proposition~\ref{prpSyzForm} we have
\[\rank \MW(\pi)\leq 10k-2.\]
This upper bound is a corollary from the Shioda-Tate formula for the Mordell-Weil rank of elliptic surfaces. Cogolludo and Libgober \cite{CogLib} used this bound to bound the degree of the Alexander polynomial.
 
In this and the next section we will give an upper bound $g(k)$ for $\rank \MW(\pi)$ such that  \[ \lim_{k\to \infty} \frac{g(k)}{k}= \frac{1}{18}\left( 125+\sqrt{73}+\sqrt{2302-106\sqrt{73}}\right) \approx 5.34.\]

\begin{definition}
Fix a positive integer $t$. Let $a_1,\dots,a_{t+1},b_1,\dots,b_t$ be a sequence of positive integers. If no confusion arises we write $\va,\vb$ for this sequence. We call $\va,\vb$ \emph{$k$-admissible for rank $2r$ and lowest degree $D_0$} if
\begin{enumerate}
\item $\sum a_i=\sum b_i$.
\item $a_i< b_i$ for $i=1,\dots,t$.
\item $a_i\geq a_{i+1}$, $b_i\geq b_{i+1}$.
\item $a_{t+1}=D_0$.
\item $b_i\leq 5k$ for $i=1,\dots,t$.
\item $\#\{i\mid b_i=5k\}\geq r$.
\item $c(\va,\vb):=\frac{1}{2}\left(\sum_{i=1}^{t+1}b_i^2-\sum_{i=1}^ta_i^2\right) \leq \min(M(6k),3ka_{t+1})$.
\end{enumerate}

We call a sequence \emph{strongly $k$-admissible} if  $a_{i}\leq  b_{i+1}$ for all $i=1,\dots t-1$.

We call a (strongly) $k$-admissible sequence \emph{reduced} if $\{a_i\} \cap \{b_i\}=\emptyset$. 
\end{definition}

\begin{remark} Given a (strongly) $k$-admissible sequence $\va,\vb$ we can construct a reduced  $k$-admissible sequence by repeatedly throwing out $b_i$ and $a_j$ in case they are equal.
The new sequence is easily to be seen $k$-admissible, since this operation does  not change $c(\va,\vb)$. Moreover, this reduction transforms a $k$-admissible sequence transforms into a $k$-admissible sequence and it transforms  a strongly $k$-admissible sequence  into a strongly $k$-admissible sequence.
\end{remark}  

The results of Sections~\ref{secRes} and~\ref{secMW} show that
\begin{lemma} Suppose that $y^2=x^3+f$ has Mordell-Weil rank $2r$. Let $D_0$ be the degree of  a generator of minimal degree of $I$. Then there exists a $k$-admissible sequence $(\va,\vb)$ of rank $2r$ and lowest degree $D_0$, such that $c(\va,\vb)=\#Z(I)$.
\end{lemma}

\begin{remark} In this section we will study strongly $k$-admissible sequences. In the next section we will show that if $y^2=x^3+f$ has rank $2r$ then there exists a strongly $k$-admissible sequence for rank $2r$ and lowest degree $D_0'\leq D_0$. Hence to prove the desired bounds we only have to consider strongly $k$-admissible sequences.
\end{remark}

For technical reasons, we have to discuss the following type of sequences separately:
\begin{lemma}\label{lemexcl} Suppose that $t=r$, $a_1=\dots=a_r=5k-1, a_{r+1}=r$ and $k\geq 2$. Then
\[ c(\va,\vb) \leq \min\left(M(6k),3kD_0\right)\]
implies $k=1$ and $r=3$.
\end{lemma}
\begin{proof}
The first two assumptions imply that
\[ c(\va,\vb)=5rk-\frac{1}{2}r(r+1)\]
and  $D_0=r$.

If $r< 4k-1$ then
\[ c(\va,\vb)> 3kr=3kD_0\]
hence we can exclude this case.

If $r\geq 4k$ then the bound for $M(6k)$ from Remark~\ref{rmkMiy} and the inequality  $5rk-\frac{1}{2}r(r+1)\leq M(6k)$ yield
\[ r\leq 5k-1/2-\frac{1}{438}\sqrt{799350k^2-287766k+47961-31974\sqrt{73}k^2+14454\sqrt{73}k)}.\]
An straight-forward calculation shows that the right hand side is strictly smaller than $4k$.


It remains to check the case $r=4k-1$. Then $c(\va,\vb)=12k^2-3k$. Now  $c(\va,\vb)\leq M(6k)$ implies that $k<2$. Hence the only case that might occur is $k=1,r=3$. \end{proof}

\begin{remark} This exceptional case $k=1,t=r=3$ does occur: 
Let $C$ be the dual of a smooth cubic curve. This is a sextic curve with 9 cusps and no further singularities. 
Since $c(\va,\vb)\leq 3kD_0$ it follows that $D_0\geq 3$.  Hence $b_i \in \{4,5\}$ and the $a_i\in \{3,4\}$.

Let $r$ be the number of $b_i$ that equals 5, let $A_4$ be the difference between the number of $i$ such that $b_i=4$ and the number of $i$ such that $a_i=4$, let $A_3$ be minus the number of $i$ such that $a_i=3$. Then we have the following three equalities
\[ r+A_4+A_3+1=0,\; 5r+4A_4+3A_3=0,\; 25r+16A_4+9A_3=18.\]
These equalities come from the following facts: there is one more $a_i$ than $b_i$; we have $\sum a_i=\sum b_i$ and we have $\sum b_i^2-\sum a_i^2=2c(\va,\vb)$.

The only solution to this system of equations is $r=3,A_4=-3,A_3=-1$. In order to actually determine the minimal resolution we need to determine $t$. If $t$ were strictly larger then $r$ then  both $a_4$ and $b_4$ equal $4$ and in particular the resolution is not minimal. Hence $t=r=3$. From this it follows that $b_1=b_2=b_3=5$, $a_1=a_2=a_3=4$ and $a_{t+1}=3$. 
Hence the exceptional case $k=1,r=3$ of the above lemma does actually occur.
\end{remark}

\begin{proposition}\label{propStrongRedHigh} Suppose $(\va',\vb')$ is strongly $k$-admissible for rank $2r$ and degree $D_0$.
Then there exists a strongly $k$-admissible sequence $(\va,\vb)$ for rank $2r$ and degree $D_0$ such that $c(\va,\vb)\leq c(\va',\vb')$ and
\begin{enumerate}
\item $k=1,r=t=3, a_1=a_2=a_3=4,a_4=3$;
\item $r=t$, there exists an integer $w$ between $0$ and $r-1$ such that  $a_1=\dots= a_w=5k-1$; $a_w>a_{w+1}\geq a_{w+2}$, $a_{w+2}=a_{w+3}=\dots=a_{r+1}=D_0$ and $b_1=\dots=b_r=5k$ or
\item there exists an integer $w$ between $0$ and $r-1$ such that $a_1=\dots =a_w=5k-1$; $a_{w+1}=\dots=a_{r+1} =D_0$; $b_1=\dots=b_r=5k$ and $b_i=a_i+1$ for $r<i\leq t$.
\end{enumerate}
\end{proposition}

\begin{proof} We start by setting $\va:=\va'$ and $\vb:=\vb'$. We apply a series of modifications to $\va,\vb$ in order to end up in one of the three above mentioned forms.

First of all we may reduce $\va,\vb$, i.e. there are no pairs $i,j$ such that $a_i=b_j$.
Moreover, from Lemma~\ref{lemexcl} it follows that $D_0<5k-1$.

Recall that $c(\va,\vb)=\frac{1}{2}\left(\sum b_i^2-\sum a_i^2\right)$.
We apply several operations on $\va,\vb$ that fix $r$ and $D_0$, keep the sequence strongly $k$-admissible and reduced and lower the function $c$:
\begin{enumerate}
\item Let $i$ be the smallest index such that $a_i<5k-1$, let $j>i$ such that $a_j>D_0$. Assume that $i<r$, hence $b_{i+1}=5k>a_i+1$. Replace in $\va$, $a_i$ by $a_i+1$ and $a_j$ by $a_j-1$. The new sequence is clearly strongly $k$-admissible (here one uses $i<r$) and has a lower value of $c(\va,\vb)$ (here one uses $a_i>a_j$).
\item If for some $r<i<t$ we have that $b_i-b_t\geq 2$ and $b_i-a_{i-1}\geq 2$ then we can decrease $b_i$ by one and increase $b_t$ by one.
\item If for some $r\leq i<t$ we have that $a_i>D_0$ then we can decrease both $a_i$ and $b_{i+1}$.
\end{enumerate}
It might be that one has to reorder the $a_i$ and $b_i$ after applying one of the above operations or that one has to reduce the sequence. 

\textbf{Step 1: adjust $\va$ such that at most one $a_i$ is different from $5k-1$ and $D_0$.}
Applying the first operation several times brings us in the situation that at most one of the $a_i$ is different from $5k-1,D_0$, or that $a_{r-1}$ equals $5k-1$.
If we are in the latter case and at least two of the $a_i$ are different from $5k-1,D_0$ then $t>r$. In this case we apply the third operation (combined with reducing and sorting if necessary) until either $t=r$ or $a_r=D_0$ holds.
Hence  we are now in the situation that at most one $a_i$ is different from $5k-1,D_0$.

\textbf{Step 2: case $t=r$}.
If $t=r$ then all the $b_i$ equal $5k$. From Lemma~\ref{lemexcl} it follows that either at least two of the $a_i$ are different from $5k-1$ or $k=1,t=r=3$ holds.
Hence we are either in the first or in the second case of the Proposition.

\textbf{Step 3: case $t\neq r$}.
Suppose now that $t>r$. Applying the third operation several times brings us in the situation where all the $a_i$ are either $5k-1$ or $D_0$. 

Suppose now that $a_r=5k-1$. Let $i$ be the largest index such that $a_i=5k-1$, let $j$ be the largest such that $b_j=5k$. Since $\va,\vb$ is strongly $k$-admissible we have that $j>i\geq r$.
Replace in $\va,\vb$ $a_i$ by $D_0$ and $b_j$ by $D_0+1$, and sort $\vb$. Then the new sequence has a lower value of $c$. Iterating this allows us to assume that $b_{r+1}<5k$ and hence that $a_r=D_0$.

Let $i$ be largest index such that $b_i\neq D_0+1$. If $i=r$ then our sequence is of the third from. 
Suppose now that  $i>r$. From $\va,\vb$ we obtain a new strongly $k$-admissible sequence of length $t+1$, by decreasing $b_i$ by one and by setting $b_{t+1}=D_0+1$, $a_{t+2}=D_0$. The new sequence has a lower value of $c$.
Iterating this yields a sequence $\va,\vb$ such that $a_i$ is either $5k-1$ or $D_0$ and $b_i$ is either $5k$ or $D_0+1$, and such that $a_r=D_0$.
\end{proof}

\begin{remark} If $(\va,\vb)$ is $k$-admissible then $c(\va,\vb)\leq \min (M(6k),3ka_{t+1})$ holds. Let $m(d)$ be the smallest integer bigger or equal than $2M(d)/d$. Then the above  mentioned condition can be rephrased as $c(\va,\vb)\leq 3ka_{t+1}$ if $a_{t+1}\leq m(6k)$ and $c(\va,\vb)\leq M(6k)$ if $a_{t+1}\geq m(6k)$.
 \end{remark}

\begin{proposition}\label{prpFirstRed} Suppose that $\va,\vb$ is a strongly $k$-admissible sequence of rank $2r$. Then $r<m(6k)$ if $k>1$ and $r\leq m(6)=3$ if $k=1$.
\end{proposition}

\begin{proof}
Suppose we have a strongly $k$-admissible sequence $\va,\vb$ with $r\geq m(6k)$. Without loss of generality we may assume that $\va,\vb$ is reduced. 

If $D_0=r$ then by Lemma~\ref{lemexcl} it follows that $k=1$ and $r=3$. If $(k,r)\neq(1,3)$ then we have that $D_0>r$, hence  at least two of the $a_i$ are different from $5k-1$. 

By decreasing $a_{t+1}$ by one and increasing $a_t$ by one, we obtain a new sequence, that is again strongly $k$-admissible: the value of $c$ decreases by this operation, and since the new $D_0$ is still larger or equal than $m(6k)$ we have that $\min(M(6k),3ka_{t+1})=M(6k)$. Since the value of $c$ for the old sequence was already smaller than this quantity the value of $c$ for this new sequence is that again. If necessary replace $\va,\vb$ by its reduction.
By iterating this we end up in the case that $a_t=5k-1$ and $D_0=r$. This is impossible by Lemma~\ref{lemexcl}.
\end{proof}

\begin{proposition} \label{prpmincusp} Suppose $\va,\vb$ is a strongly $k$-admissible sequence of rank $2r$. Then 
\[ c(\va,\vb)\geq
\frac{3k}{2} \left(r-1+2k+\sqrt{-r^2+4kr+1-4k+4k^2}\right).
\]
\end{proposition}

\begin{proof}
For fixed $r,D_0$ the minimum value of $c(\va,\vb)$ is attained by a sequence of the form described in
 Proposition~\ref{propStrongRedHigh}.
 
 The first case of Proposition~\ref{propStrongRedHigh}  occurs only for $(k,r)=(1,3)$. In this case $c(\va,\vb)=9$  holds and we have
\[ c(\va,\vb)=9=\frac{3k}{2} \left(r-1+2k+\sqrt{-r^2+4kr+1-4k+4k^2}\right).\]

In the second and third case we will vary $D_0$ and the additional parameter $w$ to determine the minimum value of $c(\va,\vb)$ for fixed $r$.

Consider now sequences of the form (2).

Suppose first that $D_0> \frac{1}{2}(5k+r-1)$. Then $w<r-1$.
Set $w'=(5rk-rD_0-D_0)/(5k-1-D_0)$. From \[B(s-a_{w+1})\leq (w-w')B(s-5k+1)+(1-w+w')B(s-D_0)
\] it follows that
\begin{small}
\begin{eqnarray*} c(\va,\vb)&=& rB(s-5k)+B(s)-wB(s-5k+1)-B(s-a_{w+1})-(r-w)B(s-D_0)\\ &\geq& rB(s-5k)+B(s)-w' B(s-5k+1)-(r+1-w')B(s-D_0)\\&=&\frac{1}{2}(5kD_0+5rk-rD_0-D_0).\end{eqnarray*} 
\end{small}

The right hand side is decreasing in $D_0$. Moreover, the right hand side is precisely $c(\va,\vb)$ for $D_0=\frac{1}{2}(5k+r-1)$. Hence the minimal value for $c(\va,\vb)$ is attained at $D_0=\frac{1}{2}(5k+r-1)$. For this value one has that
\[ c(\va,\vb)=\frac{1}{4}(1-10k+10rk+25k^2-r^2)\]
which is smaller than
\[ 3kD_0=\frac{3}{2}k(r+5k-1).\]
Hence if for any $D_0>\frac{1}{2}(5k+r-1)$ there exist a strongly admissible sequence of type $(2)$ then there exist one with $D_0=\lceil \frac{1}{2}(5k+r-1)\rceil$.

Suppose now that $D_0< \lceil \frac{1}{2}(5k+r-1) \rceil$ and hence that $w=r-1$ holds. Then 
\[ c(\va,\vb)=\frac{1}{2}r(1-r)-D_0+5kD_0-D_0^2+D_0 r.\]
We look now for the smallest $D_0$ such that $c(\va,\vb)\leq 3kD_0$. Call this value $D_{0,\min}$. For $D_0<D_{0,\min}$ our sequence $\va,\vb$ is not $k$-admissible, for $D_0>D_{0,\min}$ we find a higher value of $c(\va,\vb)$.
Solving
\[ 3kD_0=\frac{1}{2}r(1-r)-D_0+5kD_0-D_0^2+D_0 r\]
yields
\[ D_{0,\min}= \frac{1}{2} \left(2k-1+r+\sqrt{4k^2+4kr-r^2-4k+1}\right).\]
One easily checks that $D_{0,\min}<\frac{1}{2}(5k+r-1)$. Hence 
the minimal possible value of $c(\va,\vb)$ is then $3kD_{0,\min}$ which is precisely the bound mentioned in the statement.

Consider now case (3) of Proposition~\ref{propStrongRedHigh}.
 We have $b_1=\dots=b_r=5k$, $b_{r+1}=\dots=b_{t}=D_0+1$, $a_1=\dots=a_w=5k-1$ and $a_{w+1}=\dots=a_{t+1}=D_0$, $w\leq r<D_0$. In order to have $\sum a_i=\sum b_i$ we need
\[t=(5k-1-D_0)(w-r)+D_0.\]
 Note that $t$ is increasing in $w$ and the function $2 c(\va,\vb)$ equals
\[ (D_0+1-5k)(5k-2-D_0)w+D_0-5rk+D_0^2+25rk^2-10rkD_0+D_0r+D_0^2r. \]
This function is either constant or is decreasing in $w$, hence we may take $w$ as large as possible, namely $w=r-1$, and therefore
\[ t=1-5k+2 D_0.\]
 Note that $D_0\geq \frac{1}{2}(r+5k-1)$. Differentiating $c$ with respect to $D_0$ yields that $c$ increases as function of $D_0$ for $D_0\geq \frac{1}{2}(r+5k-1)$. Hence the minimal value for $c$ is attained at $D_0= \frac{1}{2}(r+5k-1)$. In that case we have
\[c(\va,\vb)=\frac{1}{4}(25k^2+2rk-10k-r^2+1).\]
A straightforward calculation shows that $3kD_{0,\min}-\frac{1}{4}(25k^2+2rk-10k-r^2+1)$ has a maximum in $r=2k$, and that for $r=2k$ this function is negative. Hence for each $D_0$ we have that $c(\va,\vb)\geq 3kD_{0,\min}$ which finishes the proof.
\end{proof}

\begin{remark} For $r=1$ we find that $C$ has at least $6k^2$ cusps. If we expand the right hand side of the inequality as a function for $k\to \infty$ we obtain
\begin{align*} \frac{3k}{2} \left(r-1+2k+\sqrt{-r^2+4kr+1-4k+4k^2}\right)&=&\\6k^2+3(r-1)k+\frac{3}{4} r(1-r)+O\left( \frac{1}{k} \right).\end{align*}
\end{remark}

\begin{remark} We can specialize to the case $k=1$. Then we find that to have $r=1$ one needs at least 6 cusps, to have $r=2$ one needs at least 8 cusps, to have $r=3$ one needs at least 9 cusps and for $r>3$ one would need at least 10 cusps. Since a sextic has at most 9 cusps, this is not possible.
 The above mentioned bounds for the minimal number of cusps are sharp, see \cite[Theorem 9.2]{elljcst}.
\end{remark}

\begin{corollary}\label{Corstrongbnd} There is a function $g:\Z_{>0}\to \Z_{>0}$ such that for each $r>g(k)$ there does not exist a strongly $k$-admissible sequence $\va,\vb$  of rank $2r$. Moreover one has for $k\geq 2$ that
\[ g(k) \leq   \frac{1}{2628}\left(-219-33\sqrt{73}+(9125+73\sqrt{73})k-\sqrt{\alpha}\right).\]
where
\[ \alpha=3325734-14454\sqrt{73}+(-11766432+287328\sqrt{73})k+(12267358-564874\sqrt{73})k^2.\]
In particular,
\[ \limsup_{k\to \infty} \frac{g(k)}{k} \leq \frac{1}{36}\left(125+\sqrt{73}-\sqrt{2302-106\sqrt{73}}\right)\approx 2.67\]
\end{corollary}

\begin{proof}
 We know that for fixed $r$ and $k$ we have by the previous Proposition that
\[ c(\va,\vb)\geq
\frac{3k}{2} \left(r-1+2k+\sqrt{-r^2+4kr+1-4k+4k^2}\right).
\]

If $\va,\vb$ is $k$-admissible then
\[ c(\va,\vb)\leq M(k)\leq \frac{125+\sqrt{73}}{12}  k^2- \frac{511+ 11\sqrt{73}}{292} k.\]
Combining both inequalities yields that
\[ r\leq  \frac{1}{2628}\left(-219-33\sqrt{73}+(9125+73\sqrt{73})k-\sqrt{\alpha}\right)\]
or
\[ r\geq  \frac{1}{2628}\left(-219-33\sqrt{73}+(9125+73\sqrt{73})k+\sqrt{\alpha}\right)\]
In the latter case $r\geq 4k-1\geq m(6k)$.
This case is excluded  by Proposition~\ref{prpFirstRed}.
\end{proof}

\section{Admissible case}\label{secAmi} 
We consider now the case  where the resolution of $I$, the ideal of the cusps, yields a sequence $\va',\vb'$ that is $k$-admissible but not strongly $k$-admissible. We show that in this case there is a curve of small degree containing many of the cusps. This yields additional numerical constraints on $\va',\vb'$. We will use these extra constraints to construct a strongly $k$-admissible sequence $\va,\vb$ for the same rank such that $c(\va,\vb)\leq c(\va',\vb')$:
\begin{proposition} \label{Prpnononstrong} Suppose $\va',\vb'$ form a sequence coming from the resolution of the ideal of cusps of a cuspidal curve for rank $2r$. Suppose that $\va',\vb'$ is not strongly $k$-admissible. Then there exists a \emph{strongly $k$-admissible} sequence $\va,\vb$  for rank $2r$ such that $c(\va,\vb)\leq c(\va',\vb')$. 
\end{proposition}
\begin{proof}
Set $\va=\va'$ and $\vb=\vb'$. We are going to modify $\va$ and $\vb$ such that they become strongly admissible.

\textbf{Step 1: Set-up and goal} 

To study non-strongly $k$-admissible sequence we need to introduce some further notation.
Without loss of generality we may assume that $\va,\vb$ is reduced, in particular $a_i\neq b_{i+1}$.

Since the $a_i,b_i$ are not strongly $k$-admissible, there is an $i$ such that $a_i>b_{i+1}$. Let $i_0$ be the smallest index such that $a_{i_0}>b_{i_0+1}$ holds. Let $A=a_{i_0}$ and $D_2=\sum_{i=1}^{i_0} b_i-a_i$.

Define
\[ h_1:=\sum_{i=1}^{i_0} B(s-b_i)-B(s-a_i) \]
and
\[ h_2:=B(s)-B(s-a_{t+1})+\sum_{i=i_0+1}^t B(s-b_i)-B(s-a_i).\]
Then $h_1+h_2$ is the Hilbert Polynomial of $I$. In particular, $\# \Sigma=h_1+h_2$.

Consider now the ideal $I'$ generated by the elements of $I$ of degree strictly less than $A$. Let $f_1$ be the greatest common divisor of the elements of $I'$. Let $I''$ be the ideal generated by the elements of $I'$ divided by $f_1$.  Let $D_1:=\deg(f_1)$.

Now $I''$ defines a zero-dimensional scheme, which is possibly  non-reduced, moreover $I''$ might be not saturated.
If $I''$ were saturated then its resolution has length 1, but otherwise  the resolution might be of length 2. From this it follows that $I'$ has a resolution of length at most $2$. Hence there might exist  an $s_0\in \Z_{\geq 0}$ and $c_j,d_j\in \Z$ for $j=1,\dots,s_0$ such that $A\leq c_j< d_j$ and
\[0\to \oplus_{j=1}^{s_0} S(-d_j) \to\oplus_{j=1}^{s_0} S(-c_j) \bigoplus  \oplus_{i>i_0} S(-b_i) \to \oplus_{i>i_0} S(-a_i) \to S\to S/I'\]
is a minimal resolution of $I'$.
Let $h_3=\sum_{j=1}^{s_0}\left( B(s-c_j)-B(s-d_j)\right)$.

With this notation we have that the Hilbert polynomial of $I$ equals $h_1+h_2$, the Hilbert Polynomial of $I'$ equals $h_2+h_3$ and the Hilbert polynomial of $I''$ equals $h_2+h_3-h_C$, with $h_C(s)=D_1s-\frac{1}{2}D_1(D_1-3)$, the Hilbert polynomial of the ideal $(f_1)$.

We will find some additional restrictions.

Note that the Hilbert polynomial of $I''$ is constant. The coefficient of $s$ in $h_2(s)+h_3(s)$ equals $D_2+\sum d_j-c_j$. Hence $D_1=D_2+\sum (d_j-c_j)$. Since $d_j>c_j$ it follows that $D_2\leq D_1$.
Using $i_0\geq r$ and $D_2=\sum_{i=1}^{i_0} (b_i-a_i)$ we obtain that $D_2\geq r$.

Suppose the curve $f_1=0$ contains more than $3kD_1$ cusps of $C$. Then $C$ and $Z(f_1)$ have a common component. Using the same reasoning as in Proposition~\ref{prpSingLocu} one can show that the common component has non-positive degree. From this it follows that $Z(f_0)$ contains at most $3kD_1$ points of $\Sigma$. Therefore $Z(I'')$ contains at least $c(\va,\vb)-3kD_1$ points of $\Sigma$ and
\[ 3kD_1\geq\# Z(I)-\# Z(I'') \geq h_1+h_2-h_2-h_3+h_C =h_1-h_3+h_C.\]

Note that up to degree $A-1$ the generators and syzygies of $I$ and $I'$ agree. Hence $h_I(A-1)=h_{I'}(A-1)$. Now $h_I(A-1)=h_2(A-1)$ and $h_{I'}=h_{I''}+h_C$. From this we get
\[ 0\leq h_{I''}(A-1)=h_{I'}(A-1)-h_C(A-1)=h_2(A-1)-h_C(A-1).\]

Summarizing we found sequences $\va,\vb,\vc,\vd$ and integers $i_0,D_0,D_1,D_2,A$ such that
(with $h_1,h_2,h_3,h_C$ as above)
\begin{enumerate}
 \item $a_i\geq A$ for $i=1,\dots, i_0$,
\item $b_i<A$ for $i=i_0+1,\dots ,t$.
\item $D_0\leq a_i<b_i\leq 5k$ for $i=1,\dots t$.
\item $b_i=5k$ for $i=1,\dots, r$. 
\item $D_0=a_{t+1}=\sum_{i=1}^t (b_i-a_i)$.
\item $A\leq c_j<d_j$ for $j=1,\dots s_0$.
\item $a_i<b_{i+1}$ for $i=1,\dots,i_0-1$
\item $D_2=\sum_{i=1}^{i_0} (b_i-a_i)$.
\item $r\leq D_2\leq D_1\leq D_0\leq A \leq 5k-1$.
\item $h_1+h_2\leq 3kD_0$.
\item $h_1+h_2\leq \frac{1}{4}(45k^2-9k)$.
\item $h_1-h_3+h_C\leq 3kD_1$.
\item\label{mainrest} $h_2(A-1)\geq h_C(A-1)$.
\end{enumerate}
We want to show that for given $r$, a sequence $\va,\vb,\vc,\vd$ and integers $D_1,D_2,A$ satisfying the above conditions there exists a sequence $\va',\vb'$ with same rank $r$, but that is strongly $k$-admissible and $c(\va,\vb)\geq c(\va',\vb')$. We do this by changing the above mentioned parameters in such a way that $c(\va,\vb)$ decreases and such that in the end we have either $D_0=A$ or $D_0=D_1=D_2$ holds. In the former case we clearly have a strongly $k$-admissible sequence. In the latter case we use that $D_0=D_2$ implies $i_0=t$, hence $a_i<b_{i+1}$ for $i=1,\dots i_0-1=t-1$, which in turn implies that the sequence is strongly $k$-admissible.

\textbf{Step 2: Optimization of $h_1,h_2,h_3$ without changing $D_0,D_1,D_2,A$ and $r$.}

We first optimize our $\va,\vb,\vc,\vd$ without changing $D_0,D_1,D_2,A$ and $r$. Specifically, we  aim at decreasing the constant coefficients of $h_1$ and $h_2$ and at increasing the constant coefficient of $h_3$. Hence at this stage we only have to consider the conditions (1)-(8).

The sequence $a_1,\dots, a_{i_0}, D_1;b_1,\dots,b_{i_0}$ is a strongly $k$-admissible sequence. This means that we can apply the same transformations as in the proof of Proposition~\ref{propStrongRedHigh}, only that we need to impose $a_i\geq A$ for $i=1,\dots, i_0$.

 If for some $i<j\leq i_0$ we have $5k-1> a_i\geq a_j>A$, then we can increase $a_i$ by one and decrease $a_j$ by one then reduce and sort. The new sequence still satisfies the above mentioned conditions, but $c(\va,\vb)$ decreases.
So $a_i\in \{5k-1,A\}$ for all but at most one $i\leq i_0$. 

Suppose $i_0>r$ and for some $i<i_0$ we have that  $A<a_i<5k-1$. Let $j$ be such that $b_j\neq 5k$. If $b_j<a_i$ then we can increase both $a_i$ and $b_j$ by $5k-1-a_i$, and hence all the $a_i\in \{5k-1,A\}$.
If $b_j>a_i$ then we can lower them both with $a_i-A$ and sort the $b_i$ if necessary. From this it follows that  $a_i\in \{5k-1,A\}$ for all $i$.

Suppose $i_0>r$ and $b_{i_0}>A+1$, then we lower $b_{i_0}$ by one and increase the length of our original sequence $\va,\vb$ by adding $A$ to $\va$ and adding $A+1$ to $\vb$. Now sort and reduce.

The optimization of $h_1$ allows us to assume for $i\leq i_0$ that either
\begin{enumerate}
 \item $i_0=r$, $b_i=5k$, $i=1,\dots r$ and there exists a $w\leq r-1$ such that $a_1=\dots a_w=5k-1>a_{w+1}\geq a_{w+2}=\dots =a_r=A$. (If $w=r-1$ then $a_{w+1}=A$. In this case we might disregard $a_{w+2}$, since $w+2>i_0$.)
\item $i_0>r$, $b_i=5k$, $i=1,\dots r$, $b_i=A+1$ for $i=r+1,\dots,i_0$ and there exists a  $w\leq r-1$ such that $a_1=\dots a_w=5k-1$, $a_{w+1}=\dots a_{i_0}=A$.
\end{enumerate}
This description of $\va,\vb$ implies $D_2\geq 5k-A+r-1$ or, equivalently, $A\geq 5k+r-1-D_2$.

Suppose we are in case (1). Then $A$ and $D_2$ determine both $w$ and $a_{w+1}$ hence the function $h_1$ depends only on $A$ and $D_2$. Denote this function by  $h_{1,a}(A,D_2)$. If we are in case (2) then $A,D_2$ and $w$ determine $i_0$. Hence we have a function $h_{1,b}(A,D_2,w)$. Now $h_{1,b}$ is decreasing in $w$, hence we may assume that $w=r-1$. Since 
\[ h_{1,b}(A,D_2,r-1)\leq h_{1,a}(A,D_2) +\frac{1}{2} (A+2-5k)(A-5k-r+d_1+1)\]
we have that the constant polynomial $h_{1,b}-h_{1,a}$ is negative. Hence we may assume that $h_1=h_{1,b}(A,D_2,r-1)$, i.e.,
\[ h_1(s)=-D_2s -r+1+5kr-rA-D_2+D_2A-\frac{15}{2}k+\frac{3}{2}A+\frac{1}{2}A^2+\frac{25}{2}k^2-5kA.\]

The optimization of $h_3$ is relatively easy. First lowering $c_j$ and $d_j$ simultaneously increases the constant coefficient of $h_3$, so we may assume that $c_j=A$ for all $j$.  Suppose that for some $j$ we have $d_j\neq A+1$. We can lower $d_j$ by one as follows: we increase the length of $\vc$ by one by setting $c_{s_0+1}=A$. We increase the length of $\vd$ by setting $d_{s_0+1}=A+1$, and decreasing $d_j$ by one. Then the constant coefficient of $h_3$ increases under this operation. This allows us to assume that $d_j=A+1$ and
\[ h_3=(D_1-D_2)s+(D_1-D_2)(1-A).\]

We can optimize $h_2$  as follows. If for some $j>i_0$ we decrease $b_j$ and $a_j$ simultaneously by one then the constant coefficient of $h_2$ decreases. If we extend $\va,\vb$ by setting $a_{t+2}=D_0$, $b_{t+1}=D_0+1$, and lowering one of the $b_j$ for some $j>i_0$ then $c(\va,\vb)$ decreases. However, we have to stop as soon as $h_2(A-1)=h_C(A-1)$. I.e., this 
allows us to assume that $h_2=\max(h_{2,a},h_{2,b})$ with
\begin{itemize}
\item $h_{2,a}=B(s)-B(s-D_0)+(D_0-D_2)(B(s-D_0-1)-B(s-D_0))=D_2s+ (\frac{1}{2}D_0^2+\frac{1}{2}D_0-D_2(D_0+1)$,
\item $h_{2,b}=D_2s+\frac{1}{2}(D_1-D_1^2)+D_1A-D_2(A-1)$.
\end{itemize}

In the next step we are going to vary $A,D_0,D_1,D_2$ in such a way that if we start with a $k$-admissible sequence, not strongly $k$-admissible and satisfying the conditions (1)-(\ref{mainrest}) then the new sequence is still $k$-admissible, satisfies (1)-(\ref{mainrest}), but might have a lower value of $c(\va,\vb)=h_1+\max(h_{2,a},h_{2,b})$. It turns out that we end up either in the case $A=D_0$ or $D_2=D_0$.

The only remaining variables are $A,D_0,D_1,D_2$, hence we may disregard the conditions (1)-(8). Also we want to minimize $h_1+h_2$, hence condition (11) is a priori fulfilled. By replacing $h_{2}$ by $\max h_{2,a},h_{2,b}$ we forced condition (13) to hold. Hence we need only to consider the conditions (9), (10) and (12). We consider (9) as describing a domain in which the parameters $A,D_0,D_1,D_2$ may vary, and try to minimize $h_1+h_2$ in such a way that (10) and (12) hold.

\textbf{Step 3: Elimination of $A$}

Note that the main conditions we are considering are
\[ h_{1}+h_{2,a} \leq 3kD_0,h_{1}+h_{2,b} \leq 3kD_0, h_{1}+h_{C}-h_3\leq 3kD_1.\]
Since $h_{2,b}+h_3-h_C$ is a constant polynomial, $h_3(A-1)=0$ and $h_{2,b}(A-1)=h_C(A-1)$ it follows that $h_{2,b}=h_C-h_3$. So the left hand side of the second and third inequality agree. Since $D_0>D_1$ we might ignore the second inequality.

Both $h_1+h_{2,a}$ and $h_1+h_{2,b}$ are increasing as a function of $A$, for $A\geq 5k+r-1-D_2$. Hence we might take $A$ as small as possible, which means that either $A=D_0$ or $A=5k+r-1-D_2$. In the first case we are done. So from now on we assume that
\[ A=5k+r-1-D_2.\]

\textbf{Step 4: Elimination of $D_1$}

Substituting $A=5k+r-1-D_2$ yields new functions $h_{1}$, $h_{2,a}$, $h_{2,b}$ and $h_3$.
The function $h_{1}+h_{2,a}$ increases with $D_0$, whereas $h_{1}+h_{2,b}$ is independent of $D_0$. So we might decrease $D_0$ until one of the following three cases occurs:
\[ D_0=D_1, h_{2,b}=h_{2,a} \mbox{ or } h_{1}+h_{2,a}=3kD_0 \]
We claim now that even in the second and in the third case we may also assume that $D_0=D_1$ holds.

If $D_0$ is such that $h_{2,b}=h_{2,a}$ then the only interesting inequalities are
\[ h_{1}+h_{2,b}\leq \min\left(3kD_0,\frac{1}{4}( 45k^2-9k)\right), \; h_{1}+h_{2,b} \leq 3kD_1\]
Now $h_{1}$ and $h_{2,b}$ are independent of $D_0$. Since $D_1<D_0$  we may simplify these  bounds to $h_{1}+h_{2,b}\leq \min(3kD_1, 45k^2-9k)$, which is completely independent of $D_0$. Hence we may decrease $D_0$ such that $D_1=D_0$ since  $c(\va,\vb)$ is increasing in $D_0$ and $D_0$ is not involved in any of the further bounds except $D_1\leq D_0$.

Suppose now that $D_0$ is such that $h_{1}+h_{2,a}=3kD_0$ and $h_{2,a}>h_{2,b}$. Then $3kD_1-h_{1}-h_{2,b}$ is increasing in $D_1$ if $D_2+D_1\geq r+2k$. Now $D_1\geq D_2\geq r$, hence this condition is automatically satisfied as soon as $r\geq 2k$. 

If $D_1+D_2\geq r+2k$ then this implies that we may increase $D_1$ until we reach $D_1=D_0$ or $h_{2,a}=h_{2,b}$. The latter case we can apply the above argument to obtain $D_0=D_1$. So if $D_1+D_2\geq r+2k$ we may assume $D_0=D_1$.

If $D_1+D_2 \leq r+2k$ and $h_{2,a}>h_{2,b}$ then $c(\va,\vb)$ is independent of $D_1$. Hence we might decrease $D_1$, since this decreases $h_{2,b}+h_{1}$ and increases $h_{2,a}-h_{2,b}$. Hence we get in the situation that $D_1=D_2$. The new function $3kD_2-h_{2,b}+h_{1}$ decreases with $D_2$, whereas $c(\va,\vb)$ increases with $D_2$. 
So we may decrease $D_2=D_1$ as long as all lower bounds for $D_2$ are satisfied. However, the only bound for $D_2$ is $D_2\geq r$. Hence we are in the situation that $D_1=D_2=r$. From this it follows that $A=5k-1$.
Substituting this in $h_{1}+h_{2,b}\leq 3kD_1$ yields
\[ -2kn+(1/2)k+(1/2)k^2\geq 0.\]
In particular $r\geq 4k-1$. Since $D_2+D_1\geq r+2k$ implies $r\leq 2k$ this case does not occur.

\textbf{Step 5}

Assume now $D_0=D_1$. Consider the following bounds
\[ h_{1}+h_{2,a}\leq 3kD_0, \; h_{1}+h_{2,b}\leq 3kD_0 \mbox{ and } h_{1}+h_C-h_3\leq 3kD_1\]
The second and third inequality coincide: as remarked before we have that $h_{2,b}=h_C-h_3$, moreover we assumed now that $D_0=D_1$.

Suppose that  $h_{2,b}\leq h_{2,a}$. Then we only need to consider the first inequality $h_{1}+h_{2,a}\leq 3kD_0$. Now $h_{1}+h_{2,a}$ increases with $D_2$ hence in this case we might decrease $D_2$ until either $D_2=r$ holds, or $h_{2,a}=h_{2,b}$ holds.

Suppose that $h_{2,b}\geq h_{2,a}$. Then we only need to consider the first inequality $h_{1}+h_{2,b}\leq 3kD_0$. Now $h_{1}+h_{2,b}$ decreases with $D_2$ hence in this case we might increase $D_2$ until either $D_2=D_1=D_0$ holds or $h_{2,a}=h_{2,b}$ holds.

So we have either $h_{2,a}=h_{2,b}$, $D_2=r$ or $D_0=D_1=D_2$. In the latter case we are done.

\textbf{Subcase $D_2=r$.} If $D_2=r$, then
\[ h_{1}+h_{2,b}=5kD_0-\frac{1}{2}D_0(D_0+1).\]
One of the conditions we need to check is that $h_{1}+h_{2,b}\leq 3kD_0$. This inequality implies that $D_0\geq 4k-1$. Now, $h_1+h_{2,b}$ is increasing in $D_0$ for $D_0<5k$. In particular,
\[ h_{1}+h_{2,b}\geq 5k(4k-1)-\frac{1}{2}(4k-1)4k = 12k^2-3k.\]
This contradicts $h_1+h_{2,b}\leq \frac{45}{4}k^2-\frac{9}{4}k$, unless $k=1$ and $D_0\geq 3$. Now if $k=1$ and $D_0\geq 3$ then $a_i\in\{3,4\}$ and $b_i\in\{4,5\}$. An easy computation shows that  $4-r$ of the $a_i$ equal 3. In particular, we have $b_1=5,b_2=4,a_1=a_2=a_3=3$, $b_1=b_2=5,a_1=4,a_2=a_3=3$ or $b_1=b_2=b_3=5,a_1=a_2=a_3=4,a_4=3$. These are all strongly $k$-admissible.

\textbf{Subcase $h_{2,a}=h_{2,b}$.}
Hence the remaining case is $h_{2,a}=h_{2,b}$. Using that  $D_0=D_1$, $A=5k+r-1-D_2$ it follows that  $h_{2,a}=h_{2,b}$ implies
\[0=(D_0-D_2)(D_0+1-r-5k+D_1).\]
Hence we have either $D_0=D_2$ or $D_0=r-1+5k-D_2=A$. In both case we are done.
\end{proof}

\begin{theorem}\label{mainThm} Suppose $C=Z(f)$ is a reduced degree $6k$ curve with only nodes and ordinary cusps as singularities. Then both twice the exponent of $t^2-t+1$ in the Alexander polynomial of $C$ and the Mordell-Weil rank of the elliptic 3-fold given by $y^2=x^3+f$  are at most
 \[ \frac{1}{1314}\left(-219-33\sqrt{73}+(9125+73\sqrt{73})k-\sqrt{\alpha}\right).\]
where
\[ \alpha=3325734-14454\sqrt{73}+(-11766432+287328\sqrt{73})k+(12267358-564874\sqrt{73})k^2.\]
\end{theorem}
\begin{proof} Let $r$ be  the Mordell-Weil rank.
By Proposition~\ref{Prpnononstrong} there exists a strongly $k$-admissible sequence of rank $r$.
From Corollary~\ref{Corstrongbnd} it follows that $r$ can be bounded by the  above mentioned quantity.
\end{proof}

\bibliographystyle{plain}
\bibliography{remke2.bib}

\begin{thebibliography}{10}

\bibitem{ACGH}
E.~Arbarello, M.~Cornalba, P.~A. Griffiths, and J.~Harris.
\newblock {\em Geometry of algebraic curves. {V}ol. {I}}, volume 267 of {\em
  Grundlehren der Mathematischen Wissenschaften [Fundamental Principles of
  Mathematical Sciences]}.
\newblock Springer-Verlag, New York, 1985.

\bibitem{CogLib}
J.~I. {Cogolludo-Agustin} and A.~{Libgober}.
\newblock {Mordell-Weil groups of elliptic threefolds and the Alexander module
  of plane curves}.
\newblock Preprint, available at \texttt{arXiv:1008.2018v1}, 2010.

\bibitem{DimBet}
A.~Dimca.
\newblock Betti numbers of hypersurfaces and defects of linear systems.
\newblock {\em Duke Math. J.}, 60:285--298, 1990.

\bibitem{Dim}
A.~Dimca.
\newblock {\em Singularities and topology of hypersurfaces}.
\newblock Universitext. Springer-Verlag, New York, 1992.

\bibitem{EisSyz}
D.~Eisenbud.
\newblock {\em The geometry of syzygies}, volume 229 of {\em Graduate Texts in
  Mathematics}.
\newblock Springer-Verlag, New York, 2005.

\bibitem{EsnMil}
H.~Esnault.
\newblock Fibre de {M}ilnor d'un c\^one sur une courbe plane singuli\`ere.
\newblock {\em Invent. Math.}, 68:477--496, 1982.

\bibitem{Hirano}
A.~Hirano.
\newblock Construction of plane curves with cusps.
\newblock {\em Saitama Math. J.}, 10:21--24, 1992.

\bibitem{ell3HK}
K.~Hulek and R.~Kloosterman.
\newblock Calculating the {M}ordell-{W}eil rank of elliptic threefolds and the
  cohomology of singular hypersurfaces.
\newblock {\em Ann. Inst. Fourier (Grenoble)}, 61:1133--1179, 2011.

\bibitem{elljcst}
R.~Kloosterman.
\newblock On the classification of rational elliptic threefolds with constant
  $j$-invariant.
\newblock Preprint, 2008.

\bibitem{Langer}
A.~Langer.
\newblock Logarithmic orbifold {E}uler numbers of surfaces with applications.
\newblock {\em Proc. London Math. Soc. (3)}, 86:358--396, 2003.

\bibitem{LibAlexArc}
A.~Libgober.
\newblock Alexander invariants of plane algebraic curves.
\newblock In {\em Singularities, {P}art 2 ({A}rcata, {C}alif., 1981)},
  volume~40 of {\em Proc. Sympos. Pure Math.}, pages 135--143. Amer. Math.
  Soc., Providence, RI, 1983.

\bibitem{MirEllThree}
R.~Miranda.
\newblock Smooth models for elliptic threefolds.
\newblock In {\em The birational geometry of degenerations (Cambridge, Mass.,
  1981)}, volume~29 of {\em Progr. Math.}, pages 85--133. Birkh\"auser Boston,
  Mass., 1983.

\bibitem{MiES}
R.~Miranda.
\newblock {\em The basic theory of elliptic surfaces}.
\newblock Dottorato di Ricerca in Matematica. ETS Editrice, Pisa, 1989.

\bibitem{ZarSyz}
O.~Zariski.
\newblock On the irregularity of cyclic multiple planes.
\newblock {\em Ann. of Math. (2)}, 32:485--511, 1931.

\end{thebibliography}

\end{document}